\documentclass[11pt]{amsart}
\usepackage{amsmath,amsthm,amssymb}
\usepackage{orcidlink}
\usepackage{hyperref}
\usepackage{color}
\usepackage{graphicx}
\usepackage{blindtext}
\numberwithin{equation}{section}
\newtheorem{theorem}{Theorem}[section]
\newtheorem{lemma}[theorem]{Lemma}
\theoremstyle{definition}
\newtheorem{proposition}[theorem]{Proposition}

\newtheorem{definition}{Definition}[section]

\theoremstyle{remark}
\newtheorem{remark}{Remark}

\usepackage[margin=3cm]{geometry}

\newcommand{\R}{\mathbb{R}}

\date{\today}
\title[Optimal spectral inequality for higher-dimensional Landau operator]{Optimal spectral inequality for the higher-dimensional Landau operator}

\author{
Sedef Özcan\,\orcidlink{0009-0005-4636-9406}
\and
Matthias T\"aufer\,\orcidlink{0000-0001-8473-2310}
}

\address{Sedef Özcan, Dokuz Eylül University, Faculty of Science, Department of Mathematics, Izmir, Turkey.}
\address{Matthias Täufer, CERAMATHS, Université Polytechnique Hauts-de-France,
Le Mont Houy 59313 Valenciennes Cedex 09,
France;
\text{Matthias.Taufer@uphf.fr}.}

\begin{document}

\begin{abstract}
We prove optimal spectral inequalities for Landau operators in full space and in arbitrary dimension.
Spectral inequalities are lower bounds on the  $L^2$-mass of functions in spectral subspaces of finite energy when integrated over a sampling set $S \subset \R^d$.
Landau operators are Schrödinger operators associated with a constant magnetic field of the form $(- \nabla + A(x))^2$ where $A$ is a -- in case of non-vanishing magnetic field -- unbounded vector potential.
Our strategy relies on so-called magnetic Bernstein estimates and analyticity, adapting an approach used by Kovrijkine in the context of the Logvinenko-Sereda theorem.
We generalize results previously only known in dimension $d = 2$.
The main difficulty in dimension $d \geq 3$ are the magnetic Bernstein inequalities which, in comparison to the two-dimensional case, lead to additional complications and require more delicate estimates.
Our results have immediate consequences for control theory, spectral theory and mathematical physics which we comment on.

\end{abstract}


	\maketitle

\section{Introduction}

Spectral inequalities of Lebeau-Robbiano type are inequalities of the form
\[
    \lVert \phi \rVert_{L^2(\R^d)}^2
    \leq
    C 
    \exp(c \sqrt{E})
    \lVert \phi
    \rVert_{L^2(S)}^2
    \quad
    \text{for all}
    \
    f \in \operatorname{Ran}
    \mathbf{1}_{(-\infty,E]}(A)
\]
where $A$ is a self-adjoint elliptic differential operator in $L^2(\R^d)$, $S \subset \R^d$ a suitable measurable sampling set and $ \operatorname{Ran}
    \mathbf{1}_{(-\infty,E]}(A)$ is the spectral subspace of energies up to $E$ corresponding to $A$.
They are also known as quantitative unique continuation principles or uncertainty relations and they play an important role in control theory of parabolic equations, spectral theory, and mathematical physics.

In control theory, a now classical framework -- initiated by Lebeau and Robbiano in~\cite{LebeauR-95} and later developed
systematically, e.g. in~\cite{Miller-04,TenenbaumT-11,LebeauL-12} -- shows that spectral inequalities imply observability estimates and hence null-controllability of the associated heat equation with interior control in $S$.
Furthermore, the constants $C$ and $c$ directly and explicitly affect the associated \emph{control cost}.
Indeed, in recent years spectral inequalities have been used to prove null-controllability of various parabolic problems in full space: The (classic) heat equation~\cite{EgidiV-18, WangWZZ-19}, certain hypoelliptic quadratic equations~\cite{BeauchardPS-17, BeauchardJPS-21}, the heat equation with Laplace-Beltrami operator on a Riemannian manifold with analytic potential which asymptotically looks like $\R^d$~\cite{LebeauM-19}, the heat equation with bounded potential~\cite{NakicTTV-20}, 
the heat equation with power-growth potentials~\cite{DickeSV-24},
Shubin operators~\cite{AlphonseS-24},
the two-dimensional Landau operator~\cite{PfeifferT-25},
and quantum graphs~\cite{EgidiMS-24}.
In spectral theory, they have been used to prove lower bounds on the effect of perturbations on spectra~\cite{NakicSTTV-20} and for semiclassical eigenvalue estimates~\cite{FrankLS-25}.
In mathematical physics, spectral inequalities have been a cornerstone in understanding quantum systems, in particular in condensed matter physics in the context of Anderson localization, see~\cite{TaeuferV-15, SeelmannT-20,CapoferriT-25} and the references therein.

There has been a particular focus in the interplay between spectral inequalities and the geometry of the sampling set $S$.
So-called \emph{thick sets} have emerged as the natural and optimal class for observability of the heat equation
\cite{EgidiV-18, WangWZZ-19} and indeed this is the case in which spectral inequalities from thick sets are best understood since spectral subspaces of the Laplacian can be formulated in terms of the support of the Fourier transform whence a spectral inequality is equivalent to from the classic Kovrijkine-Logvinenko-Sereda theorem~\cite{LogvinenkoS-74, Kovrijkine-00, Kovrijkine-01}.
Unfortunately, this approach is quite rigid and not well-suited to generalizations such as adding a generic electric potential or a magnetic field.
Therefore, it has been customary to instead work with the more restrictive geometric notion of \emph{equidistributed sets}, that are unions of open balls which are spaced in such a way that every elementary cell contains at least one such ball~\cite{NakicTTV-18,NakicSTTV-20,DuanWC-20,DickeSV-24a,DuanYZ-26}.
Corresponding spectral inequalities typically rely on Carleman estimates which require a more restrictive geometry but allow to treat bounded potentials as a perturbation of the free Laplacian.
Following a similar Carleman-type approach there also exist some works on spectral inequalities with \emph{bounded} magnetic potentials, see e.g.~\cite{TautenhahnT-18}.
However, in $\R^d$, this amounts to a magnetic field which, up to small fluctuations, is macroscopically vanishing and this excludes in particular the physically most relevant case of a homogenous constant magnetic field.
Constant magnetic fields in full space will require an inherently \emph{unbounded} magnetic potential and lower order terms can no longer be treated as perturbations.
Indeed, let \(x=(x_{1},\dots,x_{d})^{\top}\in\R^{d}\) and 
\(B\in\R^{d\times d}\) an antisymmetric matrix describing the \emph{magnetic field}.
The associated \emph{magnetic potential} and covariant gradient are
\[
A(x)=\tfrac12 Bx,
\qquad
\tilde\nabla := -i\nabla - A(x),
\]
and the magnetic Laplacian is
\[
H_B := \tilde\nabla^{2}
      = \left(-i\nabla - A(x)\right)^{2}.
\]
This operator describes a charged quantum particle subject to a homogeneous and constant magnetic field.
There is some gauge invariance and different choices of $A$ will lead to the same magnetic field $B$, see also Section~\ref{sec:preliminaries}, but in any case, as soon as the magnetic field $B$ is non-vanishing, every corresponding magnetic potential $A$ must be unbounded and linearly growing in $x$.

As long as one works on bounded domains $\Omega \subset \R^d$, the magnetic potential associated with a constant magnetic field will remain bounded which has been used to infer controllability using Carleman estimates (see, e.g., \cite{HuangKSM-19}) but such estimates inherently depend on the diameter of the domain and, for this reason, cannot be optimal and not generalizable to the full-space case.

For the two-dimensional Schrödinger operator with constant magnetic field, an optimal spectral inequality has recently been obtained in~\cite{PfeifferT-25}.
To the best of our knowledge, this work is the first to treat the magnetic Laplacian in its natural unbounded setting on $\R^d$, establishing a spectral inequality that is independent of any domain truncation and fully captures the interplay between the unbounded magnetic potential and the geometry of the thick set.
The strategy therein exploits analyticity of eigenfunctions through the
Bernstein–type (or “analytic continuation’’) method of
\cite{Kovrijkine-00}, see also~\cite{EgidiS-21}.
The key were so-called magnetic Bernstein inequalities which represented higher-order magnetic derivatives as powers of the two-dimensional Landau operator.
The fact that suitable sums of magnetic derivatives could be recombined into powers of the two-dimensional Landau operator seemed to be a somewhat lucky circumstance related to dimension $d = 2$ which did not work beyond see~\cite[Remark 6]{PfeifferT-25}.

In this paper we surmount this problem and prove sharp spectral inequalities for \(H_B\) on
\(\R^{d}\).
For every energy level \(E>0\) and every measurable set
\(S\subset\R^{d}\) that is \emph{thick} in the sense of
\cite{EgidiV-18,EgidiV-20}, we prove
\[
\|f\|_{L^{2}(\R^{d})}^{2}
 \le C(E,B,S)\,\|f\|_{L^{2}(S)}^{2},
 \qquad
 f \in \operatorname{Ran}\mathbf{1}_{(-\infty,E]}(H_B),
\]
where the observability constant \(C(E,B,S)\) is explicit and optimal in its
dependence on the energy \(E\), the geometry of \(S\), and the magnetic field
\(B\).  
The estimate recovers the known two-dimensional Landau case and provides the
first higher-dimensional result with fully explicit dependence on all
parameters.
Our key contribution is a finer analysis of the algebra generated by magnetic derivatives and estimates in which we complete sums of magnetic derivatives with non-negative terms to suitable powers of magnetic Schrödinger operators thus complementing the algebraic analysis-of-magnetic-derivatives approach in~\cite{PfeifferT-25} with analytic techniques see Section~\ref{sec:main-proof}.
Beyond the already mentioned applications to controllability of parabolic problems, spectral theory and Anderson localization, we emphasize that we now treat the physically relevant three- and four-dimensional cases~\cite{LiW-13}.

The paper is organised as follows:
Section~\ref{sec:preliminaries} recalls the spectral properties of \(H_{B}\);
Section~\ref{sec:analyticity} proves the analyticity estimates;
Section~\ref{sec:main-proof} shows the spectral inequality and its sharp
dependence on the parameters; and discusses the optimality of the exponents and
possible extensions.

\section{Preliminaries}
\label{sec:preliminaries}

\subsection{Notation}

Throughout this work we assume that the dimension is \(d \ge 3\).
Indeed, the case $d = 2$ has recently been treated in~\cite{PfeifferT-25} whereas in dimension \(d = 1\) there is no nontrivial constant magnetic field, and the corresponding non-magnetic statement reduces to the classical Logvinenko–Sereda–Kovrijkine theorem~\cite{Kovrijkine-00, Kovrijkine-01}.

For a vector \(x = (x_1, x_2, \dots, x_d) \in \R^d\), we denote by
\[
|x| := \left( \sum_{i=1}^d x_i^2 \right)^{1/2}
\quad \text{and} \quad
|x|_1 := \sum_{i=1}^d |x_i|
\]
its Euclidean norm and its \(\ell^1\)-norm, respectively.
For a matrix \(M = (M_{kl}) \in \R^{d \times d}\), we denote by
\[
\|M\|_{\mathrm{f}} := \left( \sum_{j,l} |M_{jl}|^2 \right)^{1/2},
\]
its Frobenius norm, and by
\[
\|M\|_1 := \max_{1 \le l \le d} \sum_{k=1}^d |M_{kl}|,
\]
the matrix norm, induced by the \(\ell^1\)-norm.
 The expression \( \mathrm{Vol}(S) \) refers to the \(d\)-dimensional Lebesgue measure of a measurable set \( S \subset \R^d \) and \( \mathbf{1}_S \) denotes the indicator function of a set \( S \). 
In particular, given a self-adjoint operator \( A \) and \( E \in \R \), we denote by \( \mathbf{1}_{(-\infty, E]}(A) \) the orthogonal projection onto the spectral subspace of \( A \) corresponding to energies up to \(E\).  We denote by \( \partial_i := \frac{d}{dx_i} \) the partial derivative with respect to the \(x_i\) coordinate.

The central geometric notion in this article are \emph{thick} or \emph{relatively dense} sets.
They originated in Fourier analysis in the context of the so-called Logvinenko-Sereda theorem~\cite{LogvinenkoS-74}, see also~\cite{Paneah-61, Kacnelson-73, Kovrijkine-00, Kovrijkine-01}.

	\begin{definition}
    Let $\ell=(\ell_{1},\cdots,\ell_{d})\in (0,\infty)^{d}$ and $\rho\in (0,1]$. 
    A measurable set $S \subset\R^{d}$ is \emph{$(\ell,\rho)$-thick} if, for every hyperrectangle $Q$ with side lengths $(\ell_{1},\cdots,\ell_{d})$ parallel to the axes, one has
\begin{equation*}
\operatorname{Vol}(S\cap Q)\geq \rho \operatorname{Vol}(Q).
\end{equation*}
We call $S$ \emph{thick} if it is $(\ell,\rho)$-thick for some $\ell \in (0,\infty)^d$ and $\rho > 0$.
\end{definition}

$(\ell,\rho)$-thick sets necessarily fill at least a \(\rho\)-fraction of the volume of every axis-aligned \(\ell\)-hyperrectangle in \(\R^{d}\).
They can be quite irregular and can even have non-empty interior as can be seen via constructions using positive measure Smith-Volterra Cantor sets.   

They have featured prominently in spectral inequalities in recent years~\cite{EgidiV-18, WangWZZ-19, NakicSTTV-20} and can be considered as optimal sets from which spectral inequalities can hold.
Indeed, for many classical differential operators such as Laplacians with a potential, spectral inequalities from thick sets, while possible true, are still out of reach and works usually resort to a more restrictive setting of so-called \emph{equidistributed} sets (unions of balls such that every elementary cell of $\mathbb{Z}^d$ contains exactly one such ball), see~\cite{NakicTTV-18, NakicSTTV-20}.

\subsection{The magnetic Schrödinger operator}

Dynamics in the absence of magnetic fields are governed by the negative Laplacian $-\Delta = \sum_{j = 1}^d \frac{\partial^2}{\partial xj}^2$.
In the presence of a magnetic field, this operator is replaced by a magnetic Schrödinger operator $(i \nabla + A(x))^2$ where $A \colon \R^d \to \R^d$ is a \emph{magnetic vector potential}.
Indeed, it is not the vector field $A$ that bears physical relevance but only the \emph{magnetic field $B$} given by its rotation $B = \operatorname{rot} A$ (as for an electrostatic potential where only its gradient, that is relative differences, affects dynamics). 
There is \emph{gauge invariance}: Every vector potential $A'$ with $\operatorname{rot} A' = 0$ yields the samy dynamics upon replacing $A$ by $A + A'$.
Actually, $B$ is a $2$-form, that is, it can be written as a function $B \colon \R^d \to \R^{\binom{d}{2}}$.
In particular, in dimension $d = 1$, there are no magnetic fields, in dimension $d = 2$, magnetic fields are determined by a value at each point $x \in \R^2$, in dimension $d = 3$ by a $3$-dimensional vector at each point $x \in \R^3$ and in higher dimensions by a $\R^{\binom{d}{2}}$-valued function.
    A lot can be said if the magnetic potential $A$ is \emph{bounded} and the magnetic Schrödinger operator can be considered as a perturbation of the Laplacian~\cite{TautenhahnT-18}.
 
    However, we are interested in a fundamentally different case, namely in \emph{constant magnetic fields}.
    These fields must necessarily arise from \emph{unbounded vector potentials}, so any attempt in dealing with corresponding magnetic Laplacians as a perturbation of the free Laplacian on unbounded domains will be futile.
    Indeed, another way to interpret a magnetic Schrödinger operator is that classic derivatives in the Laplacian are replaced by \emph{magnetic derivatives} which no longer commute but the commutators of which contain information on the underlying magnetic field.
    Using a particular choice of gauge namely the \emph{symmetric gauge}, this leads to the following definition.

\begin{definition}
    Let \( B \in \R^{d \times d} \) be a real antisymmetric matrix, representing a constant magnetic field, and define the associated magnetic vector potential by
\(
A(x) := \frac{1}{2} B x, \quad x \in \R^d.
\)
Then, the magnetic derivative in the \(k\)-th coordinate direction is defined by
\[
\tilde{\partial}_k := i \partial_k + A_k(x) = i \frac{\partial}{\partial x_k} + \frac{1}{2} \sum_{j=1}^d B_{kj} x_j, \quad \text{for } k = 1, \dots, d.
\]
The corresponding \emph{magnetic Laplace operator} or \emph{Landau operator} is 
\begin{equation}
    \label{eq:magnetic_Laplacian}
    H_B := \sum_{k=1}^d \tilde{\partial}_k^2 = (-i \nabla- A(x) )^2.
\end{equation}
\end{definition}
\noindent The operators $\tilde{\partial}_k$ satisfy the commutation relations
 $
 [\tilde{\partial}_k, \tilde{\partial}_l] = i B_{kl},
 $
and therefore generate a non-commutative algebra reflecting the underlying magnetic field.
    Equation~\ref{eq:magnetic_Laplacian} allows to interpret the magnetic Laplacian with constant magnetic field as a form of Laplacian, also referred to as~\emph{magnetic Laplacian}.
    We shall use this structure extensively in the remainder of the paper.
    Note that the algebra of non-commuting magnetic derivatives has a rich history and has also been studied due to its connection with the Heisenberg Lie algebra~\cite{Thangavelu-98}.
    We start by discussing the domain of the magnetic Laplacian which bears resemblance to the classic Sobolev spaces.

	\begin{definition}
Let \(m\in\mathbb N\).
For a multi–index \(\beta=(\beta_1,\dots,\beta_k)\) with \(1\le k\le m\),
set
\[
\tilde{\partial}^{\beta}:=\tilde{\partial}_{\beta_1}\cdots
\tilde{\partial}_{\beta_k}.
\]
The \emph{magnetic Sobolev space of order \(m\)} is
\[
W^{m,2}_{B}(\R^{d})=\{f\in L^{2}(\R^{d}) \mid \ \tilde{\partial}^{\beta}f \in L^{2}(\R^{d}) \ \text{ for all }\beta \in \{1,\cdots,d \}^{k} \text{ for all } k\leq m\}
\]
We write \(W^{m,2}_{B}\) when the domain is clear, and
\[
W^{\infty,2}_{B}(\R^{d})
   := \bigcap_{m\ge1} W^{m,2}_{B}(\R^{d}).
\]
\end{definition}
\noindent In simple terms, the magnetic Sobolev space \( W^{\infty,2}_B(\R^d) \) consists of all functions in \( L^2(\R^d) \) whose magnetic derivatives of all orders exist and remain square-integrable.  
We start by showing that the spectral subspace \(\operatorname{Ran} \mathbf{1}_{(-\infty, E]}(H_B)\) is contained in \( W^{\infty,2}_B(\R^d) \). For this purpose, recall the following auxiliary lemma from \cite[Lemma 2.1]{Sjoestrand-91}.
\begin{lemma}\label{sjöst}
	If $f \in L^{2}(\R^{d})$ and $H_{B}f \in W^{m,2}_{B}$, $m\in \mathbb{N}$, then $f\in W^{m+2,2}_{B}$.
\end{lemma}
\begin{lemma}\label{Lemma: Ran and inf mag diff}
	For every $E>0$, $\operatorname{Ran}\mathbf{1}_{(-\infty,E]}(H_{B})\subset W^{\infty,2}_{B}$.
    \end{lemma}
    The proof of Lemma~\ref{Lemma: Ran and inf mag diff} is given in Appendix~\ref{app:proofs}.

\begin{remark}\label{remark: decomposition}
Let us discuss the spectrum of the Landau operator, see also~\cite{Goffeng-10}.
It is well-known that in dimension $d = 2$, the Landau operator has purely discrete spectrum, consisting of infinitely degenerate eigenvalues at the Landau levels $\{B, 3B, 5B, \dots\}$ where $B > 0$ is the magnetic field strength.
In higher dimensions, the situation is more involved and, depending on the dimension, discrete or continuous spectrum can emerge.

Indeed, consider a real skew-symmetric matrix $B \in \R^{d \times d}$ and a corresponding magnetic potential $A(x) = \tfrac12 B x$. 
By the real normal form for skew-symmetric matrices \cite[Corollary 2.5.11]{HornJ-12}, we can find an orthogonal matrix $U \in O(d)$ and positive numbers $C_1, \dots, C_m > 0$, with $2m = \operatorname{rank}(B)$, such that
\[
\mathcal{C} := U^\top B U = \bigoplus_{j=1}^m
\begin{pmatrix}
0 & C_j\\
- C_j & 0
\end{pmatrix} \oplus 0_{d-2m}.
\]
Applying the unitary transformation $(\mathcal{U} f)(y) = f(Uy)$ on $L^2(\R^d)$, we see that the magnetic Laplacian
\(
H_B := (-i \nabla_x - A(x))^2
\)
is unitarily equivalent to
\(
(-i \nabla_y - A_{\mathcal{C}}(y))^2, \quad \text{where} \quad A_{\mathcal{C}}(y) = \tfrac12 \mathcal{C} y.
\)
Because $\mathcal{C}$ is block-diagonal, each nonzero $2\times 2$ block acts only on its own pair of coordinates, while the remaining $d-2m$ coordinates correspond to the nullspace of $B$ where $A_{\mathcal{C}}$ vanishes. 
More explicitly, in the $i$-th $2$-plane we have 
\[
A_{\mathcal{C},i}(y_i, y_{i+1}) = \tfrac12 C_i(-y_{i+1}, y_i), \qquad i = 1, \dots, m.
\]
This means that the magnetic field acts independently in each $2$-plane, and there is no magnetic effect in the nullspace.

As a consequence, $H_B$ decomposes as a sum of $m$ two-dimensional Landau Hamiltonians (each with field strength $C_j$) and the standard Laplacian on $\R^{d-2m}$:
\[
H_B \;\cong\; \sum_{j=1}^m H^{(2)}_{C_j} \otimes I + I \otimes (-\Delta_{\R^{d-2m}}).
\]
Each $H^{(2)}_{C_j}$ has the well-known discrete Landau levels
\(\{(2n_j + 1)C_j : n_j \in \mathbb{N}_0\}\), while the Laplacian on the nullspace contributes a continuous spectrum $[0,\infty)$. Hence, by \cite[Corollary 7.25]{Schmuedgen-12}, the spectrum of $H_B$ is
\[
\sigma(H_B) = \Bigl\{ \sum_{j=1}^m (2n_j + 1)C_j + \lvert \xi \rvert^2 : n_j \in \mathbb{N}_0,\; \xi \in \R^{d-2m} \Bigr\},
\]
and in particular
\[
E_0=\inf \sigma(H_B) = \sum_{j=1}^m C_j.
\]
If $d-2m>0$, which is in particular the case when $d$ is odd, each discrete "Landau level" is broadened into a continuous band due to the contribution of free Laplacian.

Finally, let us emphasize for later use that $B$ and $\mathcal{C}$ are unitarily equivalent and thus have the same spectral norm:
\begin{equation}\label{eqn:operator_norm_equivalence}
    \lvert B \rvert_{\mathrm{op}}
=\sup_{x\neq 0}\frac{\lvert B x \rvert}{\lvert x \rvert}
=\sup_{x\neq 0}\frac{\lvert \mathcal{C}x \rvert}{\lvert x \rvert}
=\lvert \mathcal{C} \rvert_{\mathrm{op}}.
\end{equation}
\end{remark}

\section{Results and Applications}
The main result of this article is the following sharp spectral inequality for the magnetic Laplacian $H_{B}$ on $\R^{d}$ with a constant magnetic field.

\begin{theorem}\label{Theorem: spectral inequality}
Let \( S \subset \R^d \) be an \((\ell, \rho)\)-thick set. Then there exist constants \( C_1, C_2, C_3, C_4 > 0 \), depending only on the dimension and the structure of the set, such that for every energy level \( E > 0 \) and every function \( f \in \operatorname{Ran} \mathbf{1}_{(-\infty, E]}(H_B) \), the following inequality holds:
\[
\lvert f \rvert_{L^2(\R^d)}^2 \leq 
\left( \frac{C_1}{\rho} \right)^{
C_2 + C_3 |\ell|_1 \sqrt{E} + C_4 
 |\ell|_1^2 \sqrt{ \|B^{2}\|_{1} } 
} 
\, \lvert f \rvert_{L^2(S)}^2.
\]
\end{theorem}
Theorem \ref{Theorem: spectral inequality} quantifies the unique continuation property for functions in spectral subspaces of the magnetic Laplacian: a function with energy at most \(E\) cannot be arbitrarily small on a thick set \(S\). The control constant grows exponentially with the energy scale \(\sqrt{E}\), the magnetic field strength \(\|B^2\|_1^{1/2}\), and crucially, with the square of the scale parameter \(|\ell|_1\) of the thick set.

We complement our main result with the following observations regarding the optimality of the exponent and the necessity of the geometric assumptions:
\begin{itemize}
    \item The quadratic dependence on the scale parameter \( |\ell|_1 \) in the exponent is optimal. By constructing a specific arrangement of holes and testing the inequality against the Gaussian ground state, one can demonstrate that the exponent must grow at least like \( |\ell|_1^2 \). This sharpness result is formally stated and proved in Remark~\ref{remark:optimality} (see Section~\ref{subsec:goodandbad}).
    \item The geometric thickness of the set \(S\) is not merely sufficient but necessary for the spectral inequality to hold. In Theorem~\ref{theorem:necessityofthickness}, we prove that if a spectral inequality of this type holds for all functions in the low-energy spectral subspace, then \(S\) must be a thick set. The proof relies on a contradiction argument using magnetic translations of the ground state to detect arbitrarily large "holes" in non-thick sets.
\end{itemize}
 The sharp spectral inequality established in Theorem~\ref{Theorem: spectral inequality} serves as a crucial analytic tool in several areas of mathematical physics. We now outline three specific applications where this quantitative unique continuation property yields immediate consequences:
\begin{itemize}
    \item \noindent\textbf{Null-Controllability of the Magnetic Heat Equation.}
A direct application of the spectral inequality (combined with the Lebeau--Robbiano strategy \cite{LebeauR-95}) is the null-controllability of the heat equation governed by the magnetic Laplacian.
Consider the system on $\R^d$:
\begin{equation}\label{eq:heat_control}
    \begin{cases}
    (\partial_t + H_B) u(t,x) = \mathbf{1}_S(x) f(t,x) & \text{in } (0,T) \times \R^d, \\
    u(0,x) = u_0(x) & \text{in } \R^d,
    \end{cases}
\end{equation}
where $S$ is a thick set and $f$ is the control function supported on $S$.
Theorem~\ref{Theorem: spectral inequality} implies that the system is null-controllable in any time $T>0$. That is, for any initial state $u_0 \in L^2(\R^d)$, there exists a control $f \in L^2((0,T)\times S)$ such that the solution satisfies $u(T, \cdot) \equiv 0$.
Moreover, the explicit dependence of our constants on the geometry allows for precise estimates on the cost of control, extending the results of \cite{WangWZZ-19} to general thick sets.

\item \noindent\textbf{Semiclassical Bounds and Eigenvalue Sums.}
The two-dimensional version of this spectral inequality was recently employed by Frank, Larson, and Pfeiffer \cite{FrankLS-25} to derive improved semiclassical estimates for the Landau Hamiltonian. Specifically, they established strengthened Berezin-Li-Yau and Kröger inequalities for the sum of eigenvalues on bounded domains $\Omega \subset \R^2$.
By removing the dimensional restriction, Theorem~\ref{Theorem: spectral inequality} provides the necessary analytic input to extend their approach to $\R^d$. In particular, it paves the way for proving that for bounded domains $\Omega \subset \R^d$ ($d\ge 3$), the sum of eigenvalues satisfies sharp lower bounds reflecting the magnetic geometry, correcting the standard Weyl asymptotics by a precise geometric factor related to the thickness of $\Omega$.

\item \noindent\textbf{Anderson localization for Continuous Random Operators.}
Our result provides the quantitative unique continuation estimates required for the analysis of multi-dimensional Anderson Hamiltonians with magnetic fields. 
Consider the random operator 
\[
H_\omega = H_B + V_\omega \quad \text{on } L^2(\R^d), \quad d \ge 3,
\]
where $V_\omega$ is a random alloy-type potential. A key step in the Bootstrap Multiscale Analysis (MSA) for proving Anderson localization (pure point spectrum) is the \emph{Wegner estimate}, which requires lifting the uncertainty of the potential to the spectral projections.
The scale-invariant nature of Theorem~\ref{Theorem: spectral inequality} allows one to prove that eigenvalues of the finite-volume restrictions of $H_\omega$ are sensitive to variations in the potential, a crucial ingredient for establishing localization at the bottom of the spectrum in dimensions $d \geq 3$, see the discussion and references in~\cite{PfeifferT-25}.
Note that the spectral theory of the higher-dimensional Landau operator is fundamentally more complicated since the two-dimensional operator enjoys the fact that its spectrum is pure point, a fact which does not necessarily hold in higher dimensions. 

\end{itemize}

\section{Algebraic Framework for Magnetic Derivatives and High-Order Estimates}\label{sec:analyticity}
In this section, we prove magnetic Bernstein inequalities in \( \R^d \). 
Let us comment on the strategy first.
The eventual goal is to show analyticity of functions in spectral subspaces of the magnetic Schrödinger operator.
For this, we will show convergence of their Taylor series which will require to control the $L^\infty$-norms of all of their derivatives. 
As a first step, we will control the $L^2$-norm of their \emph{magnetic derivatives}.
For the free Laplace operator, this follows from the fact that the $L^2$ norm of the sum over all derivatives of a given order can be identified as a quadratic form associated with the corresponding power of the Laplacian.
This strategy can be adapted to the two-dimensional Landau operator~\cite{PfeifferT-25}  in which case powers of the Laplacian are replaced by certain polynomials. 
But it fails a priori in higher dimension since sums over all magnetic derivatives can no longer be expressed as quadratic forms associated with polynomials of the Landau operator.
Our solution is a careful analysis and estimates of the corresponding expressions.
More precisely, even though expressions of the form
\(
\| \widetilde{\partial}_{\alpha_1} \widetilde{\partial}_{\alpha_2} \cdots \widetilde{\partial}_{\alpha_n} f \|_{L^2(\R^d)}^2,
\)
 \( \alpha_j \in \{1, \ldots, d\} \), cannot generally be written exactly in terms of \( H_B^k \), we will show that they can at least be bounded from above by explicit quantities involving \( H_B \). 

To proceed, we analyze the algebra generated by the operators \( \widetilde{\partial}_1, \ldots, \widetilde{\partial}_d \), including their commutation relations and the structure they induce.
\begin{definition}
For each \( m \in \mathbb{N} \), we define the operators \( X_m \) and \( Y_m \) by
\[
X_m := \sum_{k=1}^d \tilde\partial_k H_{B}^{m} \tilde\partial_k, \qquad
Y_m := i \sum_{j,l=1}^d  B_{jl} \tilde\partial_l H_{B}^{m} \tilde\partial_j.
\]
\end{definition}
Note that $X_m$ is symmetric and nonnegative in the sense that for all $f,g \in W_B^{\infty,2}(\R^d)$, we have
\[
    \left\langle f, X_m g \right\rangle
    =
    \left\langle X_m f, g \right\rangle,
    \quad
    \text{and}
    \quad
    \left\langle f, X_m f \right\rangle \geq 0, 
\]
and, due to skew-symmetry of $B$, the same holds for $Y_m.$

\noindent The initial values of these sequences are given by
\begin{align*}
X_0 &= \sum_{k=1}^{d} \tilde\partial_{k}^{2} = H_{B}, \\
Y_0 &= i \sum_{j,l=1}^d B_{jl} \tilde\partial_{l} \tilde\partial_{j} 
= i \sum_{j>l} B_{jl} ( \tilde\partial_{l} \tilde\partial_{j} - \tilde\partial_{j} \tilde\partial_{l} ) \\
&= i \sum_{j>l} B_{jl} (iB_{lj}) 
= \sum_{j>l} |B_{jl}|^2 = \| B \|_f^2,
\end{align*}
where in the second equality for $Y_0$, we used the fact that $B$ is skew-symmetric. 

Because each magnetic derivative \(\tilde\partial_{k}\) is self-adjoint on
\(L^{2}(\R^{d})\), integration by parts gives, for every
multi-index \(\alpha=(\alpha_{1},\dots,\alpha_{m})\),
\[
\bigl\|\tilde\partial_{\alpha_1}\cdots\tilde\partial_{\alpha_m} f\bigr\|_{L^2}^2
   = \left\langle
       f,\,
       \tilde\partial_{\alpha_m}\cdots\tilde\partial_{\alpha_1}
       \tilde\partial_{\alpha_1}\cdots\tilde\partial_{\alpha_m} f
     \right\rangle.
\]
Summing over all \(\alpha\in\{1,\dots,d\}^{m}\), we can abbreviate this as
\begin{equation}\label{Equation: MBI pre}
    \sum_{\alpha\in\{1,\dots,d\}^{m}}
   \bigl\|\tilde\partial_{\alpha_1}\cdots\tilde\partial_{\alpha_m} f\bigr\|_{L^2}^{2}
   = \bigl\langle f, R^{m}(\mathrm{Id}) f \bigr\rangle,
\end{equation}
where the operator $R$, mapping polynomials in the variables $\tilde \partial_k$ to itself is defined as
\[
R := \sum_{k=1}^{d} \tilde\partial_{k}(\,\cdot\,)\tilde\partial_{k}.
\]
In particular, \(R^{m}(\mathrm{Id})\) represents the $m$–fold
symmetrized product of the magnetic derivatives,
that is, the sum over all ordered multi–indices
\(\alpha\) of
\(\tilde\partial_{\alpha_m}\cdots\tilde\partial_{\alpha_1}
 \tilde\partial_{\alpha_1}\cdots\tilde\partial_{\alpha_m}\).
This identity captures the total \(L^{2}\)-norm of all order-\(m\)
magnetic derivatives through powers of \(R\) and serves as the starting
point for the magnetic Bernstein inequalities, which provide
quantitative bounds of these norms for
\(f\in W_{B}^{\infty,2}(\R^{d})\).
Thus, deriving an upper bound on \( R^m(\mathrm{Id}) \) leads directly to an upper bound on the left-hand side of the magnetic Bernstein inequality. In particular, if
\(
R^m(\mathrm{Id}) \leq F_m(H_B),
\)
for some function \( F_m \) of the magnetic Laplacian \( H_B \), then we obtain the operator form of the magnetic Bernstein inequality:
\[
\sum_{\alpha \in \{1, \dots, d\}^m} \| \tilde\partial_{\alpha_1} \cdots \tilde\partial_{\alpha_m} f \|_{L^2}^2 \leq \langle f, F_m(H_B) f \rangle.
\]
The key result of this section is the following Proposition~\ref{Lemma: Bound on R^m(Id)}, which shows that \( R^m(\mathrm{Id}) \) is bounded by a polynomial in \( H_B \). Building on this, Theorem~\ref{MBI} provides an explicit form of the polynomial bound, allowing us to replace \( R^m(\mathrm{Id}) \) in \eqref{Equation: MBI pre} with a polynomial in \( H_B \).

\begin{proposition}
Let \( Z_m := \frac{Y_m}{\sqrt{\|B^2\|_1}} \).
Then, for all integers \(m\ge0\),
\[
\begin{bmatrix}
 X_m\\ Z_m
\end{bmatrix}
\le
\begin{bmatrix}
 X_0 & Z_0
\end{bmatrix}
\begin{bmatrix}
 H_B & 2\sqrt{\|B^2\|_1}\\
 2\sqrt{\|B^2\|_1} & H_B
\end{bmatrix}^{m}
\]
where the inequality is understood in the sense of quadratic forms and line by line.
In particular,
\begin{equation}\label{boundonH}
\begin{split}
2R(H_B^m)=2X_{m+1}\le
&\left(H_B-\frac{||B||_f^{2}}{\sqrt{\|B^2\|_1}}\right)
   \bigl(H_B-2\sqrt{\|B^2\|_1}\bigr)^{m}\\
&+\left(H_B+\frac{||B||_f^{2}}{\sqrt{\|B^2\|_1}}\right)
   \bigl(H_B+2\sqrt{\|B^2\|_1}\bigr)^{m}.\qedhere
\end{split}
\end{equation}
\end{proposition}
\begin{proof}
 To derive recurrence relations for the sequences $X_m$ and $Y_m$, 
we begin by expanding the power $H_B^m$. 
For this expansion, we use the following commutator identity, since \([\tilde\partial_\ell,\tilde\partial_k]= i B_{\ell k}\),
\[
[H_B,\tilde\partial_k]
   = \sum_{\ell=1}^{d}
      [\tilde\partial_\ell^{2},\tilde\partial_k]
   = \sum_{\ell=1}^{d}
      \bigl(\tilde\partial_\ell[\tilde\partial_\ell,\tilde\partial_k]
           +[\tilde\partial_\ell,\tilde\partial_k]\tilde\partial_\ell\bigr)
   = \sum_{\ell=1}^{d} 2 i B_{\ell k}\,\tilde\partial_\ell .
\] 
to obtain the following recursive relations for the sequence $X_m$:
\begin{align*}
	X_m =& \sum_{k=1}^d\tilde\partial_k H_{B}^{m} \tilde\partial_k =\sum_{k=1}^d\tilde\partial_k H_{B}^{m-1}\Big( \tilde\partial_k H_B+\sum_{j=1}^d  2i B_{jk} \tilde\partial_j\Big)  \\=& X_{m-1} H_B +2Y_{m-1},
	\end{align*}
	and similarly for $Y_m$
	\begin{align*}
	Y_m =&i\sum_{j,l=1}^d  B_{jl} \tilde\partial_l H_{B}^{m} \tilde\partial_j =i\sum_{j,l=1}^d  B_{jl} \tilde\partial_l H_{B}^{m-1} \Big( \tilde\partial_j H_B+\sum_{k=1}^d  2i B_{kj} \tilde\partial_k \Big) \\
	=& Y_{m-1} H_B+ i\sum_{j,l=1}^d  B_{jl} \tilde\partial_l H_{B}^{m-1} \sum_{k=1}^d  2i B_{kj} \tilde\partial_k\\
	=& Y_{m-1} H_B-2 \sum_{k,l=1}^d \Big( \sum_{j=1}^d B_{kj} B_{jl}\Big)  \tilde\partial_l H_{B}^{m-1}  \tilde\partial_k,
	\end{align*}
	where in the last step, we changed the order of summations. The inner sum in the last line gives the $kl^{th}$ entry of $B^2$:
	\begin{align*}
Y_m	=& Y_{m-1} H_B-2 \sum_{k,l=1}^d  (B^{2}) _{kl} \tilde\partial_l H_{B}^{m-1} \tilde\partial_k\\ 
	\leq & Y_{m-1} H_B+2\max_{l}{\sum_{k=1}^d  \left| (B^{2}) _{kl}\right| } \sum_{k=1}^d   \tilde\partial_k H_{B}^{m-1} \tilde \partial_k\\
	=&Y_{m-1} H_B+2 \|B^2\|_1X_{m-1}.
\end{align*}
In the second step, the transition to inequality can be justified in the following manner. For any $f\in W^{\infty,2}_B(\R^d)$, we have	
\begin{align*}
	\left\langle  f,-2\sum_{k,l=1}^d  (B^{2}) _{kl} \tilde\partial_l H_{B}^{m-1}  \tilde\partial_k f\right\rangle \leq& \left| -2\right| \sum_{k,l=1}^d  \left| (B^{2}) _{kl}\right|   \left| \langle f,\tilde\partial_l H_{B}^{m-1}  \tilde\partial_k f\rangle\right| \\ 
	\leq& \sum_{k,l=1}^d  \left| (B^{2}) _{kl}\right| 2\left| \langle H_{B}^{(m-1)/2} \tilde\partial_l f, H_{B}^{(m-1)/2} \tilde\partial_k f  \rangle\right| .
	\end{align*} 
Using the Cauchy-Schwarz inequality and $2ab\leq a^2 +b^2$: 
	\begin{align*}
	&\left\langle  f,-2\sum_{k,l=1}^d  (B^{2}) _{kl} \tilde\partial_l H_{B}^{m-1}  \tilde\partial_k f\right\rangle \leq\sum_{k,l=1}^d  \left| (B^{2}) _{kl}\right| \cdot 2\|  H_{B}^{(m-1)/2} \tilde\partial_l f \| \cdot\|   H_{B}^{(m-1)/2} \tilde\partial_k f \| \\
	&\leq \sum_{k,l=1}^d  \left| (B^{2}) _{kl}\right|  \left[ \| H_{B}^{m/2} \tilde\partial_l f\| ^2 +\| H_{B}^{m/2} \tilde\partial_k f \| ^{2}\right] \\
	&\leq\left(  \max_{l}{\sum_{k=1}^d  \left| (B^{2}) _{kl}\right| }\right) \sum_{l=1}^d \| H_{B}^{m/2} \tilde\partial_l f\| ^2
    +\left( \max_{k}{\sum_{l=1}^d  \left| (B^{2})_{kl}\right| } \right) 
	\sum_{k=1}^d \| H_{B}^{m/2} \tilde\partial_k f \|^{2} .
	\end{align*}
	Since $B^2$ is symmetric, $ \max_{l}{\sum_{k=1}^d  \left| (B^{2}) _{kl}\right| }$ and $\max_{k}{\sum_{l=1}^d  \left| (B^{2}) _{kl}\right| }$ are equal. Hence
	\begin{align*}
\left\langle  f,-2\sum_{k,l=1}^d  (B^{2}) _{kl} \tilde\partial_l H_{B}^{m-1}  \tilde\partial_k f\right\rangle 	\leq& \left\langle  f,2\Big( \max_{l}{\sum_{k=1}^d  \left| (B^{2}) _{kl}\right| } \Big) \sum_{k=1}^d  \tilde \partial_k H_{B}^{m-1}  \tilde\partial_k f\right\rangle \\ 
	\leq& \left\langle  f,2\|B^2\|_1\sum_{k=1}^d   \tilde\partial_k H_{B}^{m-1}  \tilde\partial_k f\right\rangle 
\end{align*}
 Putting everything together we get the following inequality for vectors, to be understood as an inequality in each entry:
$$\begin{bmatrix}
	X_m & Y_m
\end{bmatrix}\leq \begin{bmatrix}
X_{m-1} & Y_{m-1}
\end{bmatrix} \begin{bmatrix}H_{B} & 2\|B^2\|_1\\
2&H_{B}\end{bmatrix}.$$
If we apply this process repeatedly we end up with
$$\begin{bmatrix}
	X_m & Y_m
\end{bmatrix}\leq \begin{bmatrix}
X_{0} & Y_{0}
\end{bmatrix} \begin{bmatrix}H_{B} & 2\|B^2\|_1\\
2&H_{B}\end{bmatrix}^{m}.$$
 We define $Z_{m}=\frac{Y_{m}}{\sqrt{\|B^2\|_1}}$ which leads to the following recursive relation:
$$\begin{bmatrix}
	X_m & Z_m
\end{bmatrix}\leq \begin{bmatrix}
X_{0} & Z_{0}
\end{bmatrix} \begin{bmatrix}H_{B} & 2\sqrt{\|B^2\|_1}\\2\sqrt{\|B^2\|_1}&H_{B}\end{bmatrix}^{m}.$$
Diagonalizing this matrix, we get a new form of this recursive formula:

\begin{align*}
	\begin{bmatrix}
	X_m & Z_m
\end{bmatrix}\leq & \begin{bmatrix}
H_B & \frac{||B||_f^{2}}{\sqrt{\|B^2\|_1}}
\end{bmatrix} \begin{bmatrix}H_{B} & 2\sqrt{\|B^2\|_1}\\2\sqrt{\|B^2\|_1}&H_{B}\end{bmatrix}^{m}\\
=&\frac{1}{2}  \begin{bmatrix}
	H_B & \frac{||B||_f^{2}}{\sqrt{\|B^2\|_1}}
\end{bmatrix}S
\begin{bmatrix}
	H_{B}+2\sqrt{\|B^2\|_1}&0\\0& H_{B}-2\sqrt{\|B^2\|_1}
\end{bmatrix}^m S,
\end{align*}
where \(S=\begin{bmatrix}  1 & 1\\ 1 &-1
\end{bmatrix}.\)
This leads to 
\begin{equation*}
\begin{split}
2R(H_B^m)=2X_{m+1}\le
&\left(H_B-\frac{||B||_f^{2}}{\sqrt{\|B^2\|_1}}\right)
   \bigl(H_B-2\sqrt{\|B^2\|_1}\bigr)^{m}\\
&+\left(H_B+\frac{||B||_f^{2}}{\sqrt{\|B^2\|_1}}\right)
   \bigl(H_B+2\sqrt{\|B^2\|_1}\bigr)^{m}.\qedhere
\end{split}
\end{equation*}
\end{proof}
\noindent In the next lemma, we make use of the recursive identity~\eqref{boundonH} to derive bounds on \( R^m(\mathrm{Id}) \).
\begin{lemma}\label{Lemma: Bound on R^m(Id)}
	For every $m\in \mathbb{N}$, we have
\begin{equation}
     R^{m}(\operatorname{Id})\leq \frac{d^m}{2^m} ( H_B +\sqrt{\|B^2\|_1}) ( H_B +3\sqrt{\|B^2\|_1}) \cdots( H_B +(2m+1)\sqrt{\|B^2\|_1}).
\end{equation}
 In particular
 \begin{equation*}
 \begin{split}
\|R^{m}(\operatorname{Id})\mathbf{1}_{(-\infty,E]} (H_B)\|=&\max\{|R^m(\operatorname{Id})(t)|: t \in \sigma(H_B)\cup (-\infty,E] \}\\\leq& \Bigl( \frac{d}{2}(E+\sqrt{\|B^2\|_1}m)\Bigr) ^m. 
 \end{split}
 \end{equation*}
\end{lemma}
\begin{proof}
We proceed by induction on \(m\), the exponent of \(R(\mathrm{Id})\). For the initial case, since $d\ge 2$ and $H_B$ is non-negative, we already have 
$$R(\mathrm{Id})=H_B\leq H_B +\sqrt{\|B^2\|_1}\le  \frac{d}{2}(H_B +\sqrt{\|B^2\|_1}).$$
Suppose that the inequality in the lemma holds for  $ m$:
\begin{equation*}
R^m (\operatorname{Id})\leq \frac{d^m}{2^m}( H_B +\sqrt{\|B^2\|_1}) ( H_B +3\sqrt{\|B^2\|_1}) \cdots( H_B +(2m-1)\sqrt{\|B^2\|_1} ) .
\end{equation*}
We will show that it holds for $m+1.$ First using the inductive hypothesis and the linearity of $R$, we have
\begin{align*}
	R(R^m (\operatorname{Id}))&\leq R \left( \frac{d^m}{2^m}( H_B +\sqrt{\|B^2\|_1} ) ( H_B +3\sqrt{\|B^2\|_1}) \cdots( H_B +(2m-1)\sqrt{\|B^2\|_1}) \right) \\
    &= \frac{d^m}{2^m}R \left( ( H_B +\sqrt{\|B^2\|_1} ) ( H_B +3\sqrt{\|B^2\|_1}) \cdots( H_B +(2m-1)\sqrt{\|B^2\|_1}) \right).
\end{align*}
Expanding the polynomial inside \(R\), we have
\begin{equation}\label{eqn:expansionofR^m}
    \bigl(H_B + \sqrt{\|B^2\|_1}\bigr) \bigl(H_B + 3\sqrt{\|B^2\|_1}\bigr) \cdots \bigl(H_B + (2m - 1)\sqrt{\|B^2\|_1}\bigr) = \sum_{j=0}^m D_{j,m} H_B^j,
\end{equation}
where \(D_{j,m}\) denote the coefficients. Using the linearity of \(R\), it follows that
\[
R\left(R^m(\mathrm{Id})\right) \leq\frac{d^m}{2^m} R\left(\sum_{j=0}^m D_{j,m} H_B^j\right) =\frac{d^m}{2^m}  \sum_{j=0}^m D_{j,m} R(H_B^j).
\]
Next, applying the estimate \eqref{boundonH} to \(R(H_B^j)\), and using that $H_B$ is non-negative, we get
\begin{align*}
R^{m+1}(\mathrm{Id}) &= R\bigl(R^m(\mathrm{Id})\bigr) \\
&\leq \frac{d^m}{2^m}  \sum_{j=0}^m D_{j,m} \frac{1}{2}\left( \left(H_B - \frac{||B||_f^{2}}{\sqrt{\|B^2\|_1}}\right) \bigl(H_B - 2\sqrt{\|B^2\|_1}\bigr)^j \right. \\
&\quad \left. +\left(H_B + \frac{||B||_f^{2}}{\sqrt{\|B^2\|_1}}\right) \bigl(H_B + 2\sqrt{\|B^2\|_1}\bigr)^j\right) \\
&\leq\frac{d^m}{2^m}  \left(H_B +\frac{||B||_f^{2}}{\sqrt{\|B^2\|_1}}\right) \sum_{j=0}^m D_{j,m} \bigl(H_B + 2\sqrt{\|B^2\|_1}\bigr)^j.
\end{align*}
Now the last operator is nothing but the operator $\sum_{j=0}^m D_{j,m}H_B^{j}$, given in \eqref{eqn:expansionofR^m},  computed at $H_B+2\sqrt{\|B^2\|_1}$:
\begin{equation}\label{eqn:iterative1}
    R^{m+1} (\operatorname{Id})\leq \frac{d^m}{2^m} \left( H_B +\frac{||B||_f^{2}}{\sqrt{\|B^2\|_1}}\right) ( H_B +3\sqrt{\|B^2\|_1}) \cdots( H_B +(2m+1)\sqrt{\|B^2\|_1}).
\end{equation}
As a final step, we will bound the $\|.\|_f$ norm of $B$ in terms of the 1-norm of $B^2$. The diagonal entries of $B^2$ can be written in the following way using the skew-symmetry of $B$:
\begin{align*}
(B^2)_{\ell \ell} &= \sum_{k=1}^d B_{\ell k} B_{k \ell} = \sum_{k=1}^d B_{\ell k} (-B_{\ell k}) = - \sum_{k=1}^d |B_{\ell k}|^2.
\end{align*}
Taking absolute value and summing over all $\ell = 1, \dots, d$:
\begin{equation*}
\sum_{\ell=1}^d |(B^2)_{\ell \ell}| = \sum_{\ell=1}^d \sum_{k=1}^d |B_{\ell k}|^2 = 2 \|B\|_f^2.
\end{equation*}
Since $|(B^2)_{\ell \ell}|$ is just one part of the sum $\|B^2\|_1$, for every $\ell$ we have
\begin{equation*}
|(B^2)_{\ell \ell}| \le \sum_{k=1}^d |(B^2)_{\ell k}| \le \|B^2\|_1.
\end{equation*}
Substituting:
\begin{align*}
2 \|B\|_f^2 &= \sum_{\ell=1}^d |(B^2)_{\ell \ell}| \le \sum_{\ell=1}^d \|B^2\|_1 = d \cdot \|B^2\|_1
\Rightarrow \|B\|_f^2 \le \frac{d}{2} \|B^2\|_1.
\end{align*}
Then \eqref{eqn:iterative1} becomes
\begin{equation}
    R^{m+1}(\operatorname{Id})\leq \frac{d^m}{2^m}  \Bigl( H_B +\frac{d}{2}\sqrt{\|B^2\|_1}\Bigr) ( H_B +3\sqrt{\|B^2\|_1}) \cdots( H_B +(2m+1)\sqrt{\|B^2\|_1}).
\end{equation}
Finally, the non-negativity of $H_B$ gives
\begin{equation}
    R^{m+1}(\operatorname{Id})\leq \frac{d^{m+1}}{2^{m+1}} ( H_B +\sqrt{\|B^2\|_1}) ( H_B +3\sqrt{\|B^2\|_1}) \cdots( H_B +(2m+1)\sqrt{\|B^2\|_1}),
\end{equation}
which completes the inductive step. Then the result follows.
\end{proof}

	\begin{theorem}\label{MBI}
	For every $E>0$ and $m \in \mathbb{N}$, we have the magnetic Bernstein inequality
	\begin{equation}
		\sum_{\alpha\in \{1,\cdots, d\}^{m}}\| \tilde{\partial}_{\alpha_{1}}\tilde{\partial}_{\alpha_{2}}\cdots\tilde{\partial}_{\alpha_{m}}f\|^{2}_{L^{2}(\R^{d})} \leq C_B(m)\| f\| ^{2}_{L^{2}(\R^{d})}, \ \ \forall  f \in \operatorname{Ran} \mathbf{1}_{(-\infty,E]}(H_{B})
	\end{equation}
	where $C_B(m)=\Bigl( \frac{d}{2}(E+\sqrt{\|B^2\|_1}m)\Bigr)^{m}$.
\end{theorem}
\begin{proof}
Let \( f \in \operatorname{Ran} \mathbf{1}_{(-\infty,E]}(H_{B}) \). By Lemma \ref{Lemma: Ran and inf mag diff}, we have \( f \in W^{\infty,2}(\R^d) \), which implies that the sum
\[
\sum_{\alpha \in \{1, \ldots, d\}^m} \| \tilde{\partial}_{\alpha_1} \tilde{\partial}_{\alpha_2} \cdots \tilde{\partial}_{\alpha_m} f \|_{L^2(\R^d)}^2
\]
is finite. 
Thus, we can perform integration by parts for the magnetic derivatives. Since \( R^m(\mathrm{Id}) \) is the sum of symmetric polynomials of degree \( m \), we obtain
\begin{align*}
\langle f, R^m(\mathrm{Id}) f \rangle 
&= \sum_{\alpha \in \{1, \ldots, d\}^m} \big\langle f, \tilde{\partial}_{\alpha_m} \tilde{\partial}_{\alpha_{m-1}} \cdots \tilde{\partial}_{\alpha_1} \tilde{\partial}_{\alpha_1} \tilde{\partial}_{\alpha_2} \cdots \tilde{\partial}_{\alpha_m} f \big\rangle \\
&= \sum_{\alpha \in \{1, \ldots, d\}^m} \| \tilde{\partial}_{\alpha_1} \tilde{\partial}_{\alpha_2} \cdots \tilde{\partial}_{\alpha_m} f \|_{L^2(\R^d)}^2.
\end{align*}
Applying Lemma \ref{Lemma: Bound on R^m(Id)}, it follows that
\[
\sum_{\alpha \in \{1, \ldots, d\}^m} \| \tilde{\partial}_{\alpha_1} \tilde{\partial}_{\alpha_2} \cdots \tilde{\partial}_{\alpha_m} f \|_{L^2(\R^d)}^2 \leq \Bigl( \frac{d}{2}(E+\sqrt{\|B^2\|_1}m)\Bigr) ^{m} \|f\|_{L^2(\R^d)}^2. \qedhere
\]
\end{proof}
\begin{remark}\label{rem:magnetic-derivatives-necessary}
In the context of Theorem~\ref{MBI}, it is essential to work with \emph{magnetic derivatives}
\(\tilde{\partial}_j\) rather than classical derivatives \(\partial_j\).
Functions in the spectral subspace \(\mathrm{Ran}\,\mathbf \mathbf{1}_{(-\infty,E]}(H_B)\) may have
\emph{unbounded} classical gradients in \(L^2(\R^d)\), while their magnetic derivatives
remain uniformly bounded.

For simplicity, assume that the dimension \(d = 2m\) is even so that the constant magnetic
field matrix \(B\) has full rank.  
The lowest Landau level is then an isolated eigenvalue with a normalized eigenfunction
\(\psi \in \ker(H_B - E_0)\).
For \(y \in \R^d\), define the \emph{magnetic translate}
\begin{equation}\label{eq:mag-trans}
f_y(x) := e^{i \chi(x,y)} \psi(x - y),
\qquad
\chi(x,y) := \tfrac12 x^\top B y.
\end{equation}
It is well known (see, e.g., \cite{Brown-64}) that the operator
\[
(T_y \varphi)(x) := e^{i \chi(x,y)} \varphi(x-y)
\]
is unitary on \(L^2(\R^d)\) and commutes with \(H_B\).
Therefore \(f_y \in \mathrm{Ran}\,\mathbf \mathbf{1}_{(-\infty,E]}(H_B)\) whenever \(\psi\) is.
Differentiating \eqref{eq:mag-trans} with respect to \(x\) gives
\[
\nabla f_y(x)
 = i \left[ \nabla_x \chi(x,y) \right] f_y(x)
   + e^{i\chi(x,y)} (\nabla \psi)(x - y).
\]
Since
\(
\nabla_x \chi(x,y) = \tfrac12 B y,
\)
we have
\[
\nabla f_y(x)
 = i \left( \tfrac12 B y \right) f_y(x)
   + e^{i\chi(x,y)} (\nabla \psi)(x - y).
\]
Taking \(L^2\)-norms and using the triangle inequality,
\[
\|\nabla f_y\|_{L^2}
 \ge \bigl\| i \tfrac12 B y \, f_y \bigr\|_{L^2}
   - \|\nabla \psi\|_{L^2}
 = \tfrac12 |B y| - \|\nabla \psi\|_{L^2}.
\]
Thus, as \(|y| \to \infty\), the classical gradient norm grows linearly with \(|y|\).

\noindent Applying the $j$-th magnetic derivative \(\tilde{\partial}_j\) to \(f_y\) yields
\begin{align*}
\tilde{\partial}_j f_y(x)
 &= i \left[ i \left( \tfrac12 B y \right) f_y(x)
        + e^{i\chi(x,y)} (\partial_j \psi)(x - y) \right]
    + A_j(x) f_y(x) \\
 &= -\tfrac12 (B y)_j f_y(x)
    + e^{i\chi(x,y)} i (\partial_j \psi)(x - y)
    + \tfrac12 (B x)_j f_y(x).
\end{align*}
Because \(-\tfrac12 B y + \tfrac12 B x = \tfrac12 B (x - y)\),
\[
\tilde{\partial}_j f_y(x)
  = e^{i\chi(x,y)} \bigl[ i \partial_j \psi + A_j(x-y)\psi \bigr](x - y)
  = e^{i\chi(x,y)} (\tilde{\partial}_j \psi)(x - y).
\]
Since magnetic translation is unitary,
\(
\|\tilde{\partial}_j f_y\|_{L^2}
  = \|\tilde{\partial}_j \psi\|_{L^2},
  j=1,\dots,d,
\)
uniformly in \(y\).
\end{remark}

\begin{theorem}[Necessity of Thickness for Higher Dimensions]\label{theorem:necessityofthickness}
Let $S\subset\mathbb R^d$ be measurable and let $E>E_0:=\inf\sigma(H_B)$. 
Suppose there exists $C>0$ such that
\[
\|f\|_{L^2(\mathbb R^d)}\le C\,\|f\|_{L^2(S)}\qquad\text{for all }f\in\operatorname{Ran}\mathbf \mathbf{1}_{(-\infty,E]}(H_B).
\]
Then $S$ is thick.
\end{theorem}

\begin{proof}
We argue by contradiction. Suppose that $S\subset \R^d$ is \emph{not thick}. Then
\[
\forall \ell>0,\ \forall \rho>0,\ \exists y\in\R^d \text{ such that } |S\cap Q_\ell(y)|<\rho\ell^d.
\]
Without loss of generality, assume $B$ is full rank (so $d=2m$); the general case follows by multiplying a Gaussian factor in the null directions of $B$. 
Let $\Psi_0\in \ker(H_B - E_0)$ be a normalized ground state, which is a Gaussian:
\begin{equation}\label{eqn:normalized_ground_state}
    \Psi_0(x) = \prod_{j=1}^m K_j \exp({-C_j(x_{2j-1}^2+x_{2j}^2)/2)},
\end{equation}
 We have $\|\Psi_0\|_{L^2(\R^d)}=1$ and there exist constants $A,\alpha>0$ (depending on $B$) such that
\(
|\Psi_0(x)| \leq A e^{-\alpha\lvert x \rvert^2} \text{ for all } x\in\R^d.
\)
For $y\in\R^d$, as in Remark \ref{rem:magnetic-derivatives-necessary}, the \emph{magnetic translate} is
\[
f_y(x) := T_y \Psi_0(x) = e^{i\chi(x,y)} \Psi_0(x-y),
\qquad \chi(x,y) := \frac12 x^\top B y.
\]
By unitarity of $T_y$ and commutation with $H_B$, we have
\[
f_y \in \mathrm{Ran}\,\mathbf{1}_{(-\infty,E]}(H_B), \qquad \|f_y\|_{L^2(\R^d)} = \|\Psi_0\|_{L^2} = 1.
\]
Since $\Psi_0$ decays as a Gaussian, for any $\varepsilon>0$ we can choose $\ell_\varepsilon>0$ such that for all $\ell\geq \ell_\varepsilon$ and all $y\in\R^d$,
\[
\int_{\R^d\setminus Q_\ell(y)} |f_y(x)|^2\,dx = \int_{\R^d\setminus Q_\ell(0)} |\Psi_0(x)|^2\,dx \leq \varepsilon.
\]
Fix $\varepsilon := \frac{1}{4C^2}$, where $C$ is the constant from the spectral inequality. Choose $\ell := \ell_\varepsilon$.
Since $S$ is not thick, for the $\ell$ chosen above and for $\rho>0$ , there exists $y\in\R^d$ such that
\(
|S\cap Q_\ell(y)| < \rho\ell^d.
\)
We decompose:
\begin{align*}
\|f_y\|_{L^2(S)}^2 
&\leq \int_{S\cap Q_\ell(y)} |f_y(x)|^2\,dx + \int_{\R^d\setminus Q_\ell(y)} |f_y(x)|^2\,dx.
\end{align*}
For the first term, note that $\|f_y\|_{L^\infty} = \|\Psi_0\|_{L^\infty} =: M$ (a constant depending only on $B$). Therefore,
\[
\int_{S\cap Q_\ell(y)} |f_y(x)|^2\,dx \leq M^2 |S\cap Q_\ell(y)| < M^2\rho\ell^d.
\]
For the second term, by our choice of $\ell$, we have
\[
\int_{\R^d\setminus Q_\ell(y)} |f_y(x)|^2\,dx \leq \varepsilon = \frac{1}{4C^2}.
\]
Thus,
\[
\|f_y\|_{L^2(S)}^2 < M^2\rho\ell^d + \frac{1}{4C^2}.
\]
Now choose $\rho$ so small that
\(
M^2\rho\ell^d < \frac{1}{4C^2}.
\)
With this choice,
\[
\|f_y\|_{L^2(S)}^2 < \frac{1}{4C^2} + \frac{1}{4C^2} = \frac{1}{2C^2}.
\]
Since $f_y\in\mathrm{Ran}\,\mathbf{1}_{(-\infty,E]}(H_B)$, the spectral inequality gives
\[
1 = \|f_y\|_{L^2(\R^d)} \leq C\|f_y\|_{L^2(S)} < C\cdot\frac{1}{\sqrt{2}C} = \frac{1}{\sqrt{2}} < 1,
\]
which is a contradiction.
\end{proof}

\noindent We write \(\beta \preceq \alpha\) if \(\beta\) is a subsequence of \(\alpha\), and denote by \(\alpha \setminus \beta\) the complementary subsequence of \(\beta\) in \(\alpha\).
We can now state the next theorem, which provides Bernstein-type inequalities (involving ordinary derivatives) for \( |f|^2 \), where \( f \in \operatorname{Ran}\mathbf{1}_{(-\infty, E]}(H_B) \).
\begin{theorem}
Let \( E \geq 0 \), \( m \in \mathbb{N} \), and let \( f \in \operatorname{Ran}   {1}_{(-\infty,E]}(H_{B}) \). Then the following holds:
\begin{equation}\label{est1}
\sum_{\alpha \in \{1, \dots, d\}^{m}} \| \partial^{\alpha} |f|^{2} \|_{L^{1}(\R^{d})} \leq C'_{B}(m) \lvert f \rvert^{2}_{L^{2}(\R^{d})},
\end{equation}
where the constant \( C'_{B}(m) \) is
\[
C'_{B}(m) = d^{\frac{dm}{2}} \Bigl( \frac{d}{2}(E+\sqrt{\|B^2\|_1}m)\Bigr)^{\frac{m}{2}}.
\]
Moreover, we have the pointwise estimate:
\begin{equation}\label{est2}
\sum_{\alpha \in \{1, \dots, d\}^{m}} \| \partial^{\alpha} |f|^{2} \|_{L^{\infty}(\R^{d})} \leq C_{\text{sob},d} \sum_{m'=m}^{m+d+1} C'_{B}(m') \lvert f \rvert^{2}_{L^{2}(\R^{d})},
\end{equation}
where \( C_{\text{sob},d} \) is a universal constant depending only on dimension \( d \).
\end{theorem}
    \begin{proof}
Let \( u, v \in \mathcal{C}^{\infty}(\R^{d}, \mathbb{C}) \) and \( x \in \R^{d} \). For each \( k \in \{1, \dots, d\} \), we have: 
\begin{align*}
\bar{v}(x)\tilde{\partial}_{k}u(x) - u(x)\overline{\tilde{\partial}_{k}v(x)}
&= i v(x) \partial_k u(x) + \bar{v}(x) \sum_{j \neq k} \frac{B_{jk}}{2} x_j u(x) \\
&\quad + i u(x) \partial_k \bar{v}(x) - u(x) \sum_{j \neq k} \frac{B_{jk}}{2} x_j \bar{v}(x) = i \partial_k (u \bar{v})(x).
\end{align*}
By induction, for any multi-index \( \alpha \in \{1, \dots, d\}^{m} \), we obtain:
\[
i^{m} \partial^{\alpha} |u|^{2}(x) = \sum_{\beta \preceq \alpha} (-1)^{m - |\beta|} \tilde{\partial}^{\beta} u(x) \, \overline{ \tilde{\partial}^{\alpha \setminus \beta} u(x) }.
\]
This leads to the estimate:
\begin{align*}
\sum_{\alpha \in \{1, \dots, d\}^{m}} \| \partial^{\alpha} |f|^{2} \|_{L^{1}(\R^{d})}
&\leq \sum_{|\alpha| = m} \sum_{\beta \preceq \alpha} \| \tilde{\partial}^{\beta} f \|_{L^{2}(\R^{d})} \, \| \tilde{\partial}^{\alpha \setminus \beta} f \|_{L^{2}(\R^{d})} \\
&= \sum_{k = 0}^{m} \binom{m}{k} \sum_{\substack{|\beta| = k \\ |\beta'| = m-k}} 
\| \tilde{\partial}^{\beta} f \|_{L^{2}(\R^d)}^2 \, \| \tilde{\partial}^{\beta'} f \|_{L^{2}(\R^d)}\\
&\leq \sum_{k = 0}^{m} \binom{m}{k} d^{m/2} \sqrt{ \sum_{|\beta| = k} \| \tilde{\partial}^{\beta} f \|_{L^{2}(\R^d)}^2
\sum_{|\beta'| = m-k} \| \tilde{\partial}^{\beta'} f \|_{L^{2}(\R^d)}^2 } \\
&\leq \sum_{k = 0}^{m} \binom{m}{k} d^{m/2} \sqrt{ C_B(k) C_B(m-k) } \lvert f \rvert_{L^{2}(\R^d)}^2
\end{align*}
Using the definition of $C_B(k)$, we have
\begin{align*}
\sum_{\alpha \in \{1, \dots, d\}^{m}} \| \partial^{\alpha} |f|^{2} \|_{L^{1}(\R^{d})}&\leq \sum_{k = 0}^{m} \binom{m}{k} d^{\frac{m}{2}}\Bigl( \frac{d}{2}(E+\sqrt{\|B^2\|_1}m)\Bigr)^{m/2} \lvert f \rvert_{L^{2}(\R^d)}^2\\
&\leq d^{\frac{dm}{2}}\Bigl( \frac{d}{2}(E+\sqrt{\|B^2\|_1}m)\Bigr)^{\frac{m}{2}} \lvert f \rvert_{L^{2}(\R^d)}^2.
\end{align*}
To prove estimate~\eqref{est2}, we apply the Sobolev embedding:
\[
\lvert f \rvert_{L^{\infty}(\R^{d})} \leq C_{\text{sob}} \lvert f \rvert_{W^{d+1,1}(\R^{d})},
\]
which yields:
\begin{align*}
\sum_{\alpha \in \{1, \dots, d\}^{m}} \| \partial^{\alpha} |f|^{2} \|_{L^{\infty}(\R^{d})}
&\leq C_{\text{sob}} \sum_{\alpha \in \{1, \dots, d\}^{m}} \| \partial^{\alpha} |f|^{2} \|_{W^{d+1,1}(\R^{d})} \\
&= C_{\text{sob}} \sum_{\alpha \in \{1, \dots, d\}^{m}} \sum_{|\beta| \leq d+1} \| \partial^{\beta} \partial^{\alpha} |f|^{2} \|_{L^{1}(\R^d)} \\
&= C_{\text{sob}} \sum_{m' = m}^{m + d + 1} \sum_{|\alpha'| = m'} \| \partial^{\alpha'} |f|^{2} \|_{L^{1}(\R^d)} \\
&\leq C_{\text{sob}} \sum_{m' = m}^{m + d + 1} C_B(m') \lvert f \rvert_{L^{2}(\R^d)}^2.\qedhere
\end{align*}
\end{proof}
\section{Analytic Foundations and Spectral Inequalities for Magnetic Laplacians}\label{sec:main-proof}
In this section we prove a spectral inequality for the high-dimensional magnetic Laplace operator, extending classical Laplacian results of Kovrijkine~\cite{Kovrijkine-00} within the abstract framework of~\cite{EgidiS-21} and incorporating the two-dimensional techniques of~\cite{PfeifferT-25}. 
The strategy of the proof is inspired by Kovrijkine’s approach for the classical Laplacian and further builds on the abstract framework developed in~\cite{EgidiS-21}. Additionally, we adapt techniques from~\cite{PfeifferT-25}, where magnetic Bernstein inequalities and spectral inequalities for the Landau operator on thick sets are established in 2 dimensions.

\subsection{Quantitative Lower Bounds via Dimension Reduction}
We now show that functions in a spectral subspace of the high-dimensional magnetic Laplacian have strong analyticity, extending smoothly into a complex neighborhood.
\begin{theorem}\label{analyticity}
Let \( E > 0 \) and  \( f \in \operatorname{Ran} \mathbf{1}_{(-\infty, E]}(H_{B}) \). Then the function \( |f|^{2} \) is real-analytic on \( \R^{d} \), and it extends to a holomorphic function \( \Phi \colon \mathbb{C}^{d} \to \mathbb{C} \) in a neighborhood of \( \R^{d} \).
\end{theorem}

\begin{proof}
By estimate~\eqref{est2}, the derivatives of order \( m \) are bounded by the largest term in the sum on the right-hand side:
\begin{align*}
\sum_{\alpha \in \{1,\dots,d\}^{m}}
   \|\partial^{\alpha}|f|^{2}\|_{L^{\infty}(\R^{d})}
&\le\\  (d+2)\,C_{\mathrm{sob}} & d^{\frac{d(m+d+1)}{2}}
   \Bigl(\frac{d}{2}(E+\sqrt{\|B^{2}\|_{1}}(m+d+1))\Bigr)^{\frac{m+d+1}{2}} \|f\|_{L^{2}(\R^{d})}^{2}.
\end{align*}
The right-hand side exhibits growth of the order \( m^{m/2} \). Since \( m^{m/2} \ll m! \) as \( m \to \infty \), there exists a constant \( C > 0 \) such that
\[
\| \partial^{\alpha} |f|^{2} \|_{L^{\infty}} \leq C^{|\alpha|+1} \alpha!.
\]
This bound ensures that the Taylor series of \( |f|^{2} \) converges locally uniformly, defining a holomorphic extension \( \Phi \) to a neighborhood of \( \R^{d} \).
\end{proof}

We next require a local lower bound for these analytic functions. To do this, for $r>0$ and $\ell=(\ell_{1},\ell_{2},\cdots,\ell_{d})\in (0,\infty)^{d}$, we denote $D(r)=\{ z \in \mathbb{C}: |z|<r\}$ and $D_{\ell}=D(\ell_{1})\times D(\ell_{2}) \times \cdots\times D(\ell_{d})$, respectively.
\begin{lemma}\label{auxlemma}
	Let $\phi: D(4+\epsilon)\rightarrow \mathbb{C}$ for some $\epsilon>0$ be an analytic function with $|\phi(0)|\geq 1$. Moreover, let $E\subset [0,1]$ be measurable with positive measure. Then
	\begin{equation}
		\sup_{t \in [0,1]} |\phi(t)|\leq \left(\frac{12}{\operatorname{Vol}(E)} \right)^{\frac{2\log M_{\phi}}{\log 2}} \sup_{t \in E}  |\phi (t)|, 
	\end{equation}
	where $M_{\phi}=\sup_{z\in D(4)}|\phi(z)|$.
\end{lemma}
For the reader’s convenience, we recommend consulting \cite[Lemma 15]{PfeifferT-25}.  Now we will prove Lemma \ref{technical}.
\begin{lemma}\label{technical}
Let \( Q \subset \R^{d} \) be a hyperrectangle with side lengths \( \ell_1, \ldots, \ell_d > 0 \), aligned with the coordinate axes. Suppose \( g: Q \rightarrow \mathbb{C} \) is a non-vanishing function that admits an analytic continuation \( G \) to the complex neighborhood \( Q + D(4\ell) \subset \mathbb{C}^d \), where \( \ell :=(\ell_1, \ldots, \ell_d )\). Then, for any measurable subset \( \omega \subset Q \) and any linear bijection \( \psi: \R^d \rightarrow \R^d \), the following inequalities hold:
\begin{align*}
\|g\|_{L^1(Q \cap \omega)} &\geq \frac{1}{2} \left( \frac{\operatorname{Vol}(\psi(Q \cap \omega))}{24 \operatorname{Vol}(\mathcal{S}) \operatorname{diam}(\psi(Q))^d} \right)^{\frac{2 \log M}{\log 2}} \cdot \frac{\operatorname{Vol}(Q \cap \omega)}{\operatorname{Vol}(Q)} \, \|g\|_{L^1(Q)} \\
&\geq \frac{1}{2} \left( \frac{\operatorname{Vol}(\psi(Q \cap \omega))}{24 \operatorname{Vol}(\mathcal{S}) \operatorname{diam}(\psi(Q))^d} \right)^{\frac{2 \log M}{\log 2} + 1} \, \|g\|_{L^1(Q)},
\end{align*}
where \( \mathcal{S} \) is the unit sphere in \( \R^d \), and
\(
M := \frac{\operatorname{Vol}(Q)}{\|g\|_{L^1(Q)}} \sup_{z \in Q + D(4\ell)} |G(z)| \geq 1.
\)
\end{lemma}
\noindent
Results akin to Lemma~\ref{technical} appear at several places in the literature, with the underlying idea tracing back to~\cite{Kovrijkine-00}.
The present proof is inspired by \cite[Lemma~14]{PfeifferT-25}, where an analogous statement is proved for \(L^2\)-norms and a linear bijection \(A\) is introduced.
This map is essential in the proof of Theorem~\ref{Theorem: spectral inequality}, as it allows for an optimization of the constants: without \(A\), the eccentricity of rectangles with side lengths \((\ell_1,\dots,\ell_d)\), more precisely, the ratio between their diameter and volume, would enter explicitly.
The use of \(A\) removes this geometric dependence, so that only \(|\ell|_1\) appears in the final estimates.

The proof of Lemma~\ref{technical} combines a dimension reduction argument with a one-dimensional estimate due to~\cite{Kovrijkine-00}, itself based on the classical Remez inequality for polynomials and Blaschke products.
\begin{proof}
	The proof rests on a dimension reduction strategy inspired by \cite{Kovrijkine-00} and \cite{PfeifferT-25}. Our main objective is to bound the volume of the "sublevel" set $W$, defined as the set of points where the function is small:
	\[
	W:=\left\{x \in Q: |g(x)| \leq C \| g\| _{L^{1}(Q)}\right\},
	\]
	for a carefully chosen constant $C > 0$. If we can demonstrate that $\operatorname{Vol}(W) \leq \frac{1}{2}\operatorname{Vol}(Q \cap \omega)$, the desired result follows immediately by integrating $|g|$ over the remaining portion of the set $Q \cap \omega$.
	
	To establish this volume bound, we start by selecting a "pivot" point $y \in Q$ where the function is relatively large, specifically, satisfying $|g(y)|\geq \|g\|_{L^{1}(Q)}/\operatorname{Vol}(Q)$. The challenge is to reduce the problem to one dimension by finding a line segment through $y$ that captures the behavior of $W$.
	
	We analyze the geometry of $W$ under the linear transformation $\psi$. By integrating in spherical coordinates centered at $\psi(y)$, we can relate the total volume of $\psi(W)$ to the lengths of its one-dimensional slices. Specifically,
	\[
	\operatorname{Vol}(\psi(W)) \leq \operatorname{Vol}(\mathcal{S}) \int_{\lvert \xi \rvert=1} \int_{0}^{\infty}\chi_{\psi(W)}(\psi(y)+s\xi ) \, s^{d-1}ds \, d\sigma(\xi).
	\]
	Since the maximal length of any such slice is bounded by $\operatorname{diam}(\psi(Q))$, an averaging argument (via the Mean Value Theorem) guarantees the existence of a specific direction $\zeta$ such that the corresponding line segment $I = \{y+s\tilde{\zeta}:s\geq 0\}\cap Q$ is "dense" in $W$. That is,
	\begin{equation}\label{density_bound}
		\frac{\operatorname{Vol}(I\cap W)}{\operatorname{Vol}(I)} \geq \frac{\operatorname{Vol}(\psi(W))}{\operatorname{Vol}(\mathcal{S}) \operatorname{diam}(\psi(Q))^d}.
	\end{equation}
	With this geometric foothold established, we turn to the analytic properties of $G$. We restrict the function to the segment $I$ by defining the univariate function $\phi(z) := \frac{\operatorname{Vol}(Q)}{\| g\| _{L^{1}(Q)}}G( y+\operatorname{Vol}(I)\tilde{\zeta}z)$ on the disk $D(4+\epsilon)$.
	
	By construction, $|\phi(0)| \geq 1$ and $M_\phi \leq M$. We can therefore apply the one-dimensional estimate from Lemma \ref{auxlemma}, taking $E$ to be the portion of the segment lying within $W$. Using the density bound \eqref{density_bound}, this yields:
	\begin{equation}\label{3}
	    \sup_{x \in W} |g(x)| \geq \left( \frac{\operatorname{Vol}(\psi(W))}{12 \operatorname{Vol}(\mathcal{S}) \operatorname{diam}(\psi(Q))^{d}}\right)^{\frac{2\log M_{\phi}}{\log 2}} \frac{\| g\| _{L^{1}(Q)}}{\operatorname{Vol}(Q)}.
	\end{equation}
	 Combining equation \eqref{density_bound} with \eqref{3} together with putting definition of the specific number $C$ gives the claim by the following argument:
		\begin{align*}
			\sup_{x\in W}|g(x)|
            \leq& \left( \frac{\operatorname{Vol}(\psi(Q\cap W))}{24 \operatorname{Vol}(\mathcal{S}) \operatorname{diam}(\psi(Q))^{d}}\right)^{\frac{2\log M_{\phi}}{\log 2}} \frac{1}{\operatorname{Vol}(Q)}\| g\| _{L^{1}(Q)}\\
			\leq & \left( \frac{\operatorname{Vol}(\psi(Q\cap W))}{2 \operatorname{Vol}(\psi(W))}\frac{\operatorname{Vol}(\psi(W))}{12 \operatorname{Vol}(\mathcal{S}) \operatorname{diam}(\psi(Q))^{d}}   \right)^{\frac{2\log M_{\phi}}{\log 2}} \frac{\| g\| _{L^{1}(Q)}}{\operatorname{Vol}(Q)}\\
			\leq &  \left( \frac{\operatorname{Vol}(Q \cap \omega)}{2\operatorname{Vol}(W)}\right)^{\frac{2\log M_{\phi}}{\log 2}}\sup_{x\in W}|g(x)|.\qedhere
		\end{align*}
\end{proof}

\subsection{Decomposition into Good and Bad Hyperrectangles} \label{subsec:goodandbad}
We consider a partition of \( \R^{d} \) using a countable family of open hyperrectangles \( (Q_j)_{j \in \mathbb{N}} \). Each hyperrectangle \( Q_j \) has side lengths \( \ell_1, \dots, \ell_d \), is aligned with the coordinate axes, and is disjoint from all others. The union \( \bigcup_{j \in \mathbb{N}} Q_j \) covers \( \R^d \) up to a set of Lebesgue measure zero.

\begin{definition}
Let \( E> 0 \), and let \( f \in \operatorname{Ran} \mathbf{1}_{(-\infty, E]}(H_B) \). A hyperrectangle \( Q_j \) is called \emph{good} if, for all \( m \in \mathbb{N} \) and all multi-indices \( \alpha \in \{1, \dots, d\}^{m} \), the following estimate holds:
\[
\| \partial^{\alpha} |f|^{2} \|_{L^{1}(Q_j)} \leq (d+2)^{m+1} C'_B(m) \lvert f \rvert^{2}_{L^{2}(Q_j)},
\]
where
\[
C'_B(m) = d^{\frac{dm}{2}} \left(\frac{d}{2} (E + \sqrt{ \|B^2\|_1 } \, m) \right)^{\frac{m}{2}}.
\]
If this condition fails, \( Q_j \) is called \emph{bad}.
\end{definition}
\noindent The following lemma states that the total \( L^{2} \)-norm of a function \( f \) over the good hyperrectangles is at least half of its total \( L^{2} \)-norm on \( \R^{d} \).
\begin{lemma}\label{bad}
Let \( E, B > 0 \), and let \( f \in \operatorname{Ran}{1}_{(-\infty, E]}(H_B) \). Then the sum of the \( L^2 \)-norms of \( f \) over the good hyperrectangles satisfies
\[
\sum_{j : Q_j \text{ good}} \lvert f \rvert^{2}_{L^{2}(Q_j)} \geq \frac{1}{2} \lvert f \rvert^{2}_{L^{2}(\R^{d})}.
\]
\end{lemma}
\begin{proof}
If a hyperrectangle $Q_j$ is bad, the definition implies that for some $m \in \mathbb{N}$ and $\alpha \in \{1, \dots, d\}^{m}$,
\[
\lvert f \rvert^{2}_{L^{2}(Q_j)} < \frac{\| \partial^{\alpha} |f|^{2} \|_{L^{1}(Q_j)}}{(d+2)^{m+1} C'_B(m)}.
\]
We sum over all bad hyperrectangles by bounding them with the sum over all possible violations (a union bound). Using the global estimate from Theorem~\ref{analyticity}, $\| \partial^{\alpha} |f|^{2} \|_{L^{1}(\R^{d})} \leq C'_B(m) \lvert f \rvert^{2}_{L^{2}(\R^{d})}$, we obtain:
\begin{align*}
\sum_{j: Q_j \text{ bad}} \lvert f \rvert^{2}_{L^{2}(Q_j)}
&\leq \sum_{m=0}^{\infty} \sum_{\alpha \in \{1, \dots, d\}^{m}} \frac{\| \partial^{\alpha} |f|^{2} \|_{L^{1}(\R^{d})}}{(d+2)^{m+1} C'_B(m)} \\
&\leq \lvert f \rvert^{2}_{L^{2}(\R^{d})} \sum_{m=0}^{\infty} \frac{d^m C'_B(m)}{(d+2)^{m+1} C'_B(m)} \\
&= \frac{\lvert f \rvert^{2}_{L^{2}(\R^{d})}}{d+2} \sum_{m=0}^{\infty} \left( \frac{d}{d+2} \right)^{m}
= \frac{1}{2} \lvert f \rvert^{2}_{L^{2}(\R^{d})}.
\end{align*}
The conclusion follows by complementing with the good hyperrectangles.
\end{proof}

	\begin{proof}[Proof of Theorem~\ref{Theorem: spectral inequality}]
 By Theorem~\ref{analyticity}, the function $|f|^{2}$ is real-analytic on 
$\R^d$ and admits an analytic continuation $\Phi$ to $\mathbb{C}^d$.
On each good rectangle $Q_j$ we will apply Lemma~\ref{technical} to 
$g = |f|^{2}$; the only input required is a uniform bound on the derivatives
of $g$ at one point of the rectangle.
We claim that there exists $x_0 \in Q_j$ such that for every multi-index
$\alpha$ of order $m$,
\begin{equation}\label{eq:deriv-point}
|\partial^\alpha |f|^2(x_0)|
\;\le\;
\frac{(2(d+2))^{m+1} d^m\, C'_B(m)}{\operatorname{Vol}(Q_j)}
\,\|f\|_{L^2(Q_j)}^2 .
\end{equation}
If this were false, the reverse inequality would hold throughout $Q_j$.
Integrating over the rectangle and using that $Q_j$ is good then leads to
\[
\|f\|_{L^2(Q_j)}^2
<
\sum_{m=0}^{\infty} \frac{1}{2^{m+1}}\,\|f\|_{L^2(Q_j)}^2
=
\|f\|_{L^2(Q_j)}^2,
\]
a contradiction. Hence such an $x_0$ exists.

For every $z$ in the polydisc $D_{(5\ell_1,\dots,5\ell_d)}$, Taylor’s
formula around $x_0$ together with \eqref{eq:deriv-point} gives
\[
|\Phi(z)|
\;\le\;
2(d+2)
\,\frac{\|f\|_{L^2(Q_j)}^2}{\operatorname{Vol}(Q_j)}
\sum_{m=0}^\infty
\frac{\bigl(10(d+2)d|\ell|_1\bigr)^m\,C'_B(m)}{m!} .
\]
Let $\psi$ be an affine map sending $Q_j$ to the unit rectangle.  
Lemma~\ref{technical} applied to $g=|f|^2$, $\omega = Q_j \cap S$, linear bijection $A$ that maps each $Q_j$ to the unit rectangle,
and $G=\Phi$ then gives
\[
\|f\|_{L^1(Q_j\cap S)}^2
\;\ge\;
\frac12
\left(\frac{C_1}{\rho}\right)^{\!\frac{2\log M_\Phi}{\log 2}}
\,
\|f\|_{L^1(Q_j)}^2 .
\]
Recalling
\(
C'_B(m)=d^{\frac{m}{2}(d+1)}\Bigl(\tfrac12(E+\beta m)\Bigr)^{\frac m2},
\)
we obtain
\begin{align*}
M_\Phi
&\le 2(d+2)\sum_{m=0}^{\infty}
\frac{\bigl(10(d+2)\,d\,|\ell|_{1}\bigr)^{m}\,
d^{\frac{dm}{2} }\,
\bigl(\tfrac d2(E+\beta m)\bigr)^{\frac m2}}{m!}.
\end{align*}
Using the inequality
\(
(E+\beta m)^{\frac m2}
\le 2^{\frac m2}\bigl(E^{\frac m2}+(\beta m)^{\frac m2}\bigr),
\)
we infer
\begin{align*}
M_\Phi
&\le 2(d+2)\sum_{m=0}^{\infty}
\frac{\Bigl(10\sqrt2(d+2)\,d^{\frac{d+1}{2}}\,|\ell|_{1}\Bigr)^m
E^{\frac m2}}{m!}
\\
&\quad
+2(d+2)\sum_{m=0}^{\infty}
\frac{\Bigl(10\sqrt2(d+2)\,d^{\frac{d+1}{2}}\,|\ell|_{1}\,
\beta^{1/2}\sqrt m\Bigr)^m}{m!}.
\end{align*}
The first series equals an exponential, while for the second we use
\[
\sum_{m=0}^{\infty}\frac{(s\sqrt m)^m}{m!}
\le \exp(2s^2+s), \qquad s\ge0.
\]
Hence,
\begin{align*}
M_\Phi
&\le 2(d+2)\exp\!\Bigl(
10\sqrt2(d+2)\,d^{\frac{d+1}{2}}\,|\ell|_{1}\,\sqrt E
\Bigr)
\\
&\quad\times
\exp\!\Bigl(
10\sqrt2(d+2)\,d^{\frac{d+1}{2}}\,|\ell|_{1}\,\beta^{1/2}
+200(d+2)^2d^{d+1}|\ell|_1^2\beta
\Bigr).
\end{align*}
Combining the exponentials and taking logarithms yields
\[
\ln M_\Phi
\le
\ln\!\bigl(2(d+2)\bigr)
+ A_1\,|\ell|_{1}\,\sqrt E
+ A_2\,|\ell|_{1}\,\beta^{1/2}
+ A_3\,|\ell|_{1}^{2}\,\beta,
\]
with
\[
A_1=A_2=10\sqrt2(d+2)\,d^{\frac{d+1}{2}},
\qquad
A_3=200(d+2)^2\,d^{d+1},
\qquad
\beta=\sqrt{\|B^2\|_1}.
\]
Finally, summing over all good hyperrectangles and using Lemma~\ref{bad}, we obtain
\begin{align*}
\lvert f \rvert^2_{L^1(\R^d)} 
&\le 2 \sum_{\substack{j:\\ Q_j \text{ good}}} \lvert f \rvert^2_{L^1(Q_j)} \\
&\le \sum_{\substack{j:\\ Q_j \text{ good}}}
4\left(\frac{C_1}{\rho}\right)^{\!
C_2 + C_3 |\ell|_1 \sqrt{E} 
+ C_4 \bigl( 
|\ell|_1 \sqrt[4]{\|B^{2}\|_{1}} 
+ |\ell|_1^2 \sqrt{\|B^{2}\|_{1}} 
\bigr)} 
\lvert f \rvert^2_{L^2(Q_j \cap S)} \\
&\le 4\left(\frac{C_1}{\rho}\right)^{\!
C_2 + C_3 |\ell|_1 \sqrt{E} 
+ C_4 \bigl( 
|\ell|_1 \sqrt[4]{\|B^{2}\|_{1}} 
+ |\ell|_1^2 \sqrt{\|B^{2}\|_{1}} 
\bigr)} 
\lvert f \rvert^2_{L^2(S)}.
\end{align*}
To complete the argument, we must control the magnetic contributions in the exponent by the energy cut-off \(E\). 
It suffices to bound \(\sqrt{\|B^{2}\|_{1}}\) in terms of \(\inf\sigma(H_{B})\).
The magnetic matrix \(B\in\R^{d\times d}\) is antisymmetric and, by Remark~\ref{remark: decomposition}, can be written as the block–diagonal matrix $\mathcal{C}$.
By \eqref{eqn:operator_norm_equivalence}, the operator norms of \(B\) and \(\mathcal C\) coincide.
Since all norms on finite-dimensional spaces are equivalent, there exists \(c_d>0\) (depending only on \(d\)) such that
\begin{equation}\label{eq:norm-equiv}
\|B^{2}\|_{1}\le c_d\,\|B^{2}\|_{\mathrm{op}}.
\end{equation}
For the block-diagonal form \(\mathcal C\), the eigenvalues of \(\mathcal C^2\) are \(-C_j^2\), so
\[
\|B^{2}\|_{\mathrm{op}}
=\|\mathcal C^{2}\|_{\mathrm{op}}
=\max_{1\le j\le m}C_j^{2}.
\]
Combining this with \eqref{eq:norm-equiv} yields
\begin{equation}\label{eq:B2-1norm}
\sqrt{\|B^{2}\|_{1}}
\;\le\; \sqrt{c_d}\,\max_{1\le j\le m} C_j.
\end{equation}
Since the lowest Landau level satisfies 
\(\inf\sigma(H_{B})=\sum_{j=1}^{m}C_{j}\ge\max_{j}C_{j}\),
we obtain
\begin{equation}\label{eq:B2-1norm-E}
\sqrt{\|B^{2}\|_{1}}
\;\le\; \sqrt{c_d}\,\inf\sigma(H_B)
\;\le\; \sqrt{c_d}\,E.
\end{equation}
Inserting this bound into the previous estimate shows that
\[
\lvert f \rvert_{L^2(\R^d)}^2 \le 
\left( \frac{C_1}{\rho} \right)^{
C_2 + \widetilde C_3 |\ell|_1 \sqrt{E} + \widetilde C_4  |\ell|_1^2 E \sqrt{||B||_1} }
\, \lvert f \rvert_{L^2(S)}^2,
\]
where \(\widetilde C_3,\widetilde C_4\) depend only on \(d\) and on the geometric data of the thick set.
This completes the proof.
\end{proof}
 
\begin{remark}\label{remark:optimality}
We explain why a quadratic dependence on \( |\ell|_1 \) multiplied by a magnetic scale is unavoidable in the exponent of Theorem~\ref{Theorem: spectral inequality}.
Assume first that \(B\) has full rank \(2m=d\); the case \(\operatorname{rank} B<d\) follows by adding a harmless Gaussian envelope in the null directions, which does not affect the tail behaviour below.

Let $\Psi$ denote the normalized ground state as in \eqref{eqn:normalized_ground_state}
and let \(C_{\min}:=\min_{1\le j\le m} C_j>0\).
Fix $\rho\in(0,1)$ and choose $L>2r>0$ such that the periodic hole set
\[
S^c=\bigcup_{k\in\mathbb{Z}^d} B_r(Lk), 
\qquad S=\R^d\setminus S^c,
\]
satisfies $\rho = 1-\omega_d (r/L)^d$, 
where $\omega_d = |B_1(0)| = \pi^{d/2}/\Gamma(\tfrac{d}{2}+1)$ 
denotes the volume of the unit ball in $\R^d$. 
For an axis-parallel vector $\ell=(\ell_1,\dots,\ell_d)$ with $\ell_i\asymp L$, 
the set $S$ is $(\ell,\rho)$–thick and $r\asymp L\asymp|\ell|_1$.
Placing the origin at the centre of a hole and setting \(f=\Psi_0\), we obtain
\[
\|f\|_{L^2(S)}^2
  = 1-\!\int_{S^c}|\Psi_0|^2
  \le 1-\!\int_{|z|<r}|\Psi_0|^2
  = \int_{|z|\ge r}|\Psi_0|^2\,dz.
\]
A standard Gaussian tail estimate gives
\[
\int_{|z|\ge r}|\Psi_0|^2\,dz
  \le C\,r^{d-2}\exp\!\Big(-\tfrac{C_{\min}}{4}\,r^2\Big).
\]
Since \(r\asymp|\ell|_1\), we obtain
\[
\|f\|_{L^2(S)}^2
   \le C'\,|\ell|_1^{\,d-2}\exp\!\big(-c\,C_{\min}|\ell|_1^2\big),
\]
for some constants \(C',c>0\) depending only on \(d\) and \(\rho\).
Applying the spectral inequality to \(f\) with \(\|f\|_{L^2(\R^d)}=1\) yields
\[
\Big(\tfrac{C_1}{\rho}\Big)^{\,C(\ell,E_0,B)}
   \gtrsim
   \exp\!\big(c\,C_{\min}|\ell|_1^2\big),
\]
Taking logarithms, we find
\[
C(\ell,E_0,B) \log\!\Big(\frac{C_1}{\rho}\Big) \geq \frac{c}{2}\,C_{\min}|\ell|_1^2 - \frac{d-2}{2}\log|\ell|_1 - \log C'^{1/2}.
\]
Thus, for large $|\ell|_1$, the exponent $C(\ell,E_0,B)$ must grow at least like $C_{\min}|\ell|_1^2/\log(C_1/\rho)$. In particular, a term quadratic in $|\ell|_1$ is unavoidable.
\end{remark}

\appendix
\section{Proofs of the Regularity Lemma}\label{app:proofs}

\begin{proof}[Proof of Lemma~\ref{Lemma: Ran and inf mag diff}]
Fix \(E>0\) and let \(P_{E}:=\mathbf \mathbf{1}_{(-\infty,E]}(H_{B})\).
Take \(f\in \operatorname{Ran} P_{E}\); then \(f=P_{E}f\) and \(f\in L^{2}(\R^{d})\).
Because $P_{E}$ commutes with $H_{B}$ and the spectrum of
$H_{B}$ on $\operatorname{Ran} P_{E}$ lies in $[0,E]$, we have
\[
H_{B}^{j}f \in L^{2}(\R^{d}),
\qquad
\|H_{B}^{j}f\|_{L^{2}}\le E^{j}\|f\|_{L^{2}}
\quad (j=0,1,2,\dots).
\]
We prove by downward induction on $k$ that
\[
H_{B}^{k}f \in W_{B}^{2(N-k),2}(\R^{d})
\qquad (k=0,1,\dots,N).
\]
For $k=N$ this is clear since $H_{B}^{N}f\in L^{2}$.
Assume it holds for $k+1$.  
Because $H_{B}^{k+1}f\in W_{B}^{2(N-k-1),2}$,  
Lemma~\ref{sjöst} (the Sjöstrand elliptic estimate) applied to
$u=H_{B}^{k}f$ with $m=2(N-k-1)$ gives
\[
H_{B}^{k}f \in W_{B}^{2(N-k-1)+2,2}
           = W_{B}^{2(N-k),2}.
\]
Thus the induction closes.
Taking $k=0$ we obtain $f\in W_{B}^{2N,2}$ for every $N\in\mathbb N$.
Hence
\[
f\in \bigcap_{N\ge1} W_{B}^{2N,2}(\R^{d})
    = W_{B}^{\infty,2}(\R^{d}). \qedhere
\]
\end{proof}


\end{document}